\documentclass[a4paper,12pt]{amsart}

\usepackage{amsmath, amssymb, euscript, amscd}
\usepackage{graphicx}
\usepackage{hyperref}

\newtheorem{assumption}{Assumption}

\newtheorem{theorem}{Theorem}
\newtheorem{proposition}{Proposition}

\title[On cohomological equations for Vershik automorphisms]{On cohomological equations for suspension flows over  \\Vershik automorphisms}

\author{Dmitry Zubov}
\thanks{This work is supported by the Russian Science Foundation under grant
  14-50-00005 and performed in Steklov Mathematical Institute of Russian
  Academy of Sciences.}

\address{Department of Differential Equations, Faculty of Mechanics and Mathematics,  Moscow State University,  Leninskiye Gory,  GSP-1,  119991, Moscow}
\email{dmitry.zubov.93@gmail.com}

\begin{document}

\begin{abstract}
  In this paper we give sufficient conditions for existence of bounded solution of cohomological equation for suspension flows over automorphisms of Markov compacta, which were introduced by Ito~\cite{Ito'78} and Vershik \cite{Vershik'82}.  This result can be regarded as a symbolic analogue of results due to Forni \cite{Forni'97} and Marmi, Moussa and Yoccoz \cite{MMY'05} for translation flows and interval exchange transformations. 
\end{abstract}

\maketitle

\section{Introduction}
\subsection{Markov compacta.\\ }  
	\text{   }Consider a directed graph $\Gamma$ with $2m$ vertices arranged in two levels, the top and the bottom, with $m$ vertices at each one.  Each edge in the graph $\Gamma$ goes from the vertex of the top level to the one of the bottom level, and there may be multiple edges.  We shall assume that each vertex has an edge that either starts or ends in it.  Let us denote by $\mathfrak{G}$ the set of all such graphs.

	The graph $\Gamma$ can be uniquely defined by the incidence matrix $ A = A(\Gamma)$ whose elements are defined by the formula
$$
	A_{ij}(\Gamma):=\#\{\text{edges in $\Gamma$}: i \to j\}, \text{ }i,j = 1,...,m.      
$$ 

	Now consider a sequence $\{A_n, n \in \mathbb{Z}\}$ of $(m \times m)$ incidence matrices. It defines a graded graph $\Gamma_\infty = \cup_{n \in \mathbb{Z}} \Gamma_n$.  The space $X$ of all paths in $\Gamma_\infty$ is called a \textit{Markov compactum}.  Each point in $X$ is a path in $\Gamma_\infty$, that is a sequence of edges $\{x_n, n \in \mathbb{Z}\}$ such that the terminal vertex $F(x_n)$ of each edge $x_n$ coincides with the initial vertex $I(x_{n-1})$ of $x_{n-1}$. 

	Following notation of Bufetov ~\cite{Bufetov'09}, ~\cite{Bufetov'13}, for $x \in X, n \in \mathbb{Z}$ we introduce the sets
$$ \gamma_n^+(x) = \{ x \in X : x_t = x'_t, t \ge n\} ,\text{  }\gamma_n^-(x) = \{x \in X : x_t = x'_t, t \le n\},  $$
$$ \gamma_\infty^+(x) = \bigcup_{n \in \mathbb{Z}} \gamma_n^+(x), \text{  }\gamma_\infty^-(x) = \bigcup_{n \in \mathbb{Z}} \gamma_n^-(x).$$
	The sets $\gamma_\infty^+(x)$ are the leaves of the \textit{vertical} foliation $\EuScript{F}^+(X)$ and the sets $\gamma_\infty^-(x)$ are the leaves of the \textit{horizontal} foliation $\EuScript{F}^-(X)$. 
	Note also that the sets $\gamma_n^+(x), \gamma_n^-(x)$ form semi-rings $\mathfrak{C}^+(X)$ and $\mathfrak{C}^-(X)$ respectively.  
	
	Following  Vershik \cite{Vershik'82} and Ito \cite{Ito'78}, we shall define a partial ordering $\mathfrak{o}$, which is called a \textit{Vershik ordering}, on the Markov compactum $X$. 

	For each vertex $v$ on the top level of the graph $\Gamma_n \subset \Gamma_\infty$ let us fix some linear ordering on the set of all edges of $\Gamma_n$ starting at $v$.  Now we say that $x<x'$ for $x,x' \in X$, if there exists $n \in \mathbb{Z}$ such that $x_t = x'_t$ for all $t>n$ and $x_n < x'_n$ (note that the edges $x_n$ and $x_n$ are comparable since they start at the same vertex). One can see that $x,x' \in X$ are comparable with respect to $\mathfrak{o}$ if and only if they both belong to the same leaf of the foliation $\EuScript{F}^+$.

	Similarly, if we fix a linear ordering on the set of edges ending at each given vertex of $\Gamma_n$, this ordering will induce a partial ordering $\tilde{\mathfrak{o}}$ on $X$, and $x,x' \in X$ are comparable with respect to $\tilde{\mathfrak{o}}$ if and only if they both belong to the same leaf of the foliation $\EuScript{F}^-$. The ordering $\tilde{\mathfrak{o}}$ is called a \textit{reversed Vershik ordering}.

\subsection{Renormalization cocycle and invariant measure. \\}
	\text{ } Now consider the space $\Omega = \mathfrak{G}^{\mathbb{Z}}$ of bi-infinite sequences $\{\omega_n, n \in \mathbb{Z}\}$ of graphs.  For $\omega \in \Omega$ we denote $X(\omega)$ the corresponding Markov compactum. 
A shift $\sigma$ acts on the space $\Omega$ so that $(\sigma \omega)_n = \omega_{n+1}$, and this action admits a \textit{renormalization cocycle} $\mathbb{A}(n,x)$ which is defined for $n > 0$ by the product of incidence matrices:
$$
	\mathbb{A}(n,\omega) = A(\omega_n)... A(\omega_1).
$$
	If all matrices $A(\omega_n), n \in \mathbb{Z}$ are invertible, then we can define the cocycle $\mathbb{A}(n,x)$ for negative $n$ by the following formula:
$$
	\mathbb{A}(-n, \omega) = A^{-1}(\omega_{-n})...A^{-1}(\omega_0).
$$
We set $\mathbb{A}(0,\omega)$ to be the identity matrix. 

\begin{assumption} { } \end{assumption} 
	Let us assume that there exists an ergodic $\sigma$-invariant probability measure $\mathbb{P}$ on $\Omega$ such that
\begin{itemize}
	\item There exists a graph $\Gamma_0 \in \mathfrak{G}$ such that all entries of $A(\Gamma_0)$ are positive 				and $\mathbb{P}(\{\omega: \omega_0 = \Gamma_0\} )>0$
	\item The matrices $A(\omega_n)$ are $\mathbb{P}-$almost surely invertible
	\item Both $\log (1+||\mathbb{A}||)$ and $\log (1+||\mathbb{A}^{-1}||)$ are integrable with respect to $\mathbb{P}$.
\end{itemize}

The \textit{transpose} cocycle $\mathbb{A}^t$ is associated with the reversed shift $\sigma^{-1}$ and is defined by the formula (for $n>0$)
$$
\begin{matrix}
\mathbb{A}^t(n,\omega) = A^t(\omega_{1-n})...A^t(\omega_0),\\
\mathbb{A}^t(-n,\omega) = {(A^t)}^{-1}(\omega_n)...{(A^t)}^{-1}(\omega_1).	
\end{matrix}
$$
In the same way, $\mathbb{A}^t(0,\omega)$ is set to be the identity matrix.  We suppose that Assumption 1 holds for the cocycle $\mathbb{A}^t$ too.

	The first part of  Assumption 1 implies the construction of invariant probability measure on generic (with respect to $\mathbb{P}$) individual Markov compactum. Indeed, for a full measure set of Markov compacta there exists  $n_0 = n_0(\omega) \in \mathbb{N}$ such that all entries of the matrix $\mathbb{A}(n_0,\omega)$ are positive.  The unique positive $\sigma\text{-additive}$ measure $\Phi_1^+$ on $\mathfrak{C}^+(X(\omega))$ corresponds to the unique eigenvector with positive coordinates. In a similar way, one can define the unique (up to scaling) positive $\sigma\text{-additive}$ measure $\Phi_1^-$ on the semi-ring $\mathfrak{C}^-(X(\omega))$. 

	Thus for $\mathbb{P}$-almost every Markov compactum $X = X(\omega)$ we obtain a $\sigma$-invariant probability measure $\nu = \nu_{\omega}$ which is defined for a cylindrical set $C \subset X$, that is the set of the form $\{x: x_{n+1} = e_1,..., x_{n+k} = e_k\}$, by the formula
\begin{equation}\label{nu}
\nu(C) = \Phi_1^+(\gamma_n^+(x) \cap C) \cdot \Phi_1^-(\gamma_{n+k}^-(x) \cap C). 
\end{equation}

	 The second and the third parts of Assumption 1 imply that Oseledets Multiplicative Ergodic Theorem (see \cite{KatokHasselblatt'95}, \cite{Oseledets'68},\cite{Pesin'77}) is applicable to the cocycle $\mathbb{A}$ (as well as for the transpose cocycle $\mathbb{A}^t$). Namely, for $\mathbb{P}$-almost all Markov compacta $X(\omega)$ there exists an $\mathbb{A}$-invariant direct-sum decomposition	$\mathbb{R}^m = E_\omega^u \oplus E_\omega^{cs},$ where $E_\omega^u$ is the strictly expanding subspace of $\mathbb{A}(\cdot, \omega)$ and $E_\omega^{cs}$ is its central-stable subspace.  The space $E_\omega^u$ corresponds to positive Lyapunov exponents, while the space $E_\omega^{cs}$ corresponds to zero and negative Lyapunov exponents, and the growth of vectors in $E_\omega^{cs}$ is at most sub-exponential.

\subsection{Finitely-additive measures. \\} \text{ }
	Following Bufetov (see \cite{Bufetov'09}, \cite{Bufetov'13}), we introduce the space of finitely-additive measures  $\mathfrak{B}^+(X), \mathfrak{B}^-(X)$ on the semi-rings $\mathfrak{C}^+, \mathfrak{C}^-$ corresponding to the foliations $\EuScript{F}^+, \EuScript{F}^-$ of the Markov compactum $X=X(\omega)$.  A \textit{finitely-additive measure} $\Phi^+ \in \mathfrak{B}^+(X)$ is a real-valued (not obligatory positive!) functional, defined on the sets of the form $\gamma_n^+(x)$ so that 
\begin{itemize}
\item $\Phi^+(\gamma_n^+(x)) = \Phi^+(\gamma_n^+(y))$ if  the terminal vertices $F(x_n), F(y_n)$ of edges $x_n, y_n$ coincide
\item there exist constants $\theta>0$ and $C>0$ such that \\ $\left|\Phi^+(\gamma_{-n}^+(x))\right| \le C e^{-\theta n}$ for any $x \in X$ and sufficiently large $n$. 
\end{itemize}
	We shall call the sets of the form $\gamma_n^+(x)$ \textit{Markovian arcs}.

	By the definintion, the measures $\Phi^+$ are invariant under changing of `future', so they are holonomy-invariant with respect to the foliation $\EuScript{F}^-$.	Due to this holonomy invariance, at each level $n \in \mathbb{Z}$ we can choose `canonical arcs' $\gamma_n^+(x_{i,n}) = \gamma_{i,n}^+, i\in \{1,...,m\}$ and introduce a sequence of $m\text{-dimensional}$ vectors $v_n$   
by setting
\begin{equation}\label{fam}
{(v_n)}_i = \Phi^+(\gamma_n^+(x_{i,n})).
\end{equation}
	The finite-additivity of $\Phi^+$ yields that the sequence $v_n, n \in \mathbb{Z}$ satisfies the equality $v_{n+1} = A_n v_n$.  Thus if Assumption 1 holds, then for $\mathbb{P}$-almost every Markov compactum $X$ there is an isomorphism between the space of finitely-additive measures $\mathfrak{B}^+(X)$ and the strictly expanding subspace $E_\omega^u$ of the cocycle $\mathbb{A}$. 

	In the same way one can define measures $\Phi^-\in\mathfrak{B}^-(X)$ such that
\begin{itemize}
\item $\Phi^-(\gamma_n^-(x)) = \Phi^-(\gamma_n^-(y))$ if  the initial vertices $I(x_n), I(y_n)$ of edges $x_n, y_n$ coincide
\item there exist constants $\theta>0, C>0$ such that\\
 $\left|\Phi^-(\gamma_{n}^-(x))\right| \le C e^{-\theta n}$ for any $x \in X$ and sufficiently large $n$. 
\end{itemize}
	Again, if Assumption 1 holds, there is an isomorphism between the space $\mathfrak{B}^-(X)$ and the strictly expanding subspace $\tilde{E}_\omega^u$ of the transpose cocycle $\mathbb{A}^t$. 
\subsection{ Duality and the functional $\mathbf{m}_{\Phi^-}$. \\} \text{ }
	In conditions of Assumption 1 the spaces $\mathfrak{B}^+(X)$ and $\mathfrak{B}^-(X)$ are dual to each other for $\mathbb{P}$-almost every Markov compactum $X = X(\omega)$. More precisely, for $\mathbb{P}\text{-almost all}$ Markov compacta the inner product in $\mathbb{R}^m$ induces a non-degenerate pairing $\langle \cdot, \cdot \rangle$ between the unstable subspaces of $\mathbb{A}$ and $\mathbb{A}^t$, $E_\omega^u$ and $\tilde{E}_\omega^u$ respectively, by the formula
\begin{equation}\label{dual}
\begin{matrix}
\langle v, \tilde{v} \rangle = \sum\limits_{j=1}^m v_j \tilde{v}_j, & \text{where } v \in E_\omega^u,  \tilde{v} \in \tilde{E}_\omega^u.
\end{matrix}
\end{equation}
The pairing $\langle \cdot, \cdot \rangle$ between $E_\omega^u$ and $\tilde{E}_\omega^{u}$ is identified with the one between $\mathfrak{B}^+(X)$ and $\mathfrak{B}^-(X)$ by the isomorphism described in the previous section.

	The Oseledets Multiplicative Ergodic Theorem implies that any finitely-additive measure $\Phi^+$ admits a decomposition into base functionals $\Phi_1^+,...,\Phi_d^+$ defined by the formula
\begin{equation}\label{base}
	\Phi^+ = \langle \Phi^+, \Phi^-_1 \rangle \Phi_1^+ + ... + \langle \Phi^+, \Phi^-_d \rangle \Phi_d^+,
\end{equation}
where the functionals $\Phi_1^-,..., \Phi_d^-$ form the basis in $\mathfrak{B}^-(X)$, and this basis is dual to $\{\Phi_1^+,...,\Phi_d^+\}$.

	Given the pairing $\langle \cdot, \cdot \rangle$, for arbitrary finitely-additive measure $\Phi^- \in \mathfrak{B}^-(X)$ we shall define a functional $\mathbf{m}_{\Phi^-} = \langle \Phi_1^+ , \Phi^- \rangle$ on $X$. 
This functional is a finitely-additive measure on subsets of $X$ which can be described in terms of formula \eqref{nu} in the following way: for a cylindrical set $\{x: x_{n+1} = e_1,..., x_{n+k} = e_k\} = C \subset X$ we put 
$$
	\mathbf{m}_{\Phi^-}(C) \equiv \Phi_1^+ \times \Phi^- (C) = \Phi_1^+(\gamma_n^+(x) \cap C) \cdot \Phi^-(\gamma_{n+k}^-(x) \cap C).
$$
\subsection{Flow $h_t^+$ and cohomological equation. \\}	\text{ }

	Given a Markov compactum $X$ endowed with Vershik ordering $\mathfrak{o}$, following Ito \cite{Ito'78}, we construct \textit{a vertical} flow $h_t^+, t \in \mathbb{R}$ along the leaves of the foliation $\EuScript{F}^+(X)$ by setting
$$
	\Phi_1^+([x,h_t^+x]) = t,
$$ 
where the interval $[x,h_t^+(x)]$ is considered with respect to Vershik ordering $\mathfrak{o}$. It is not hard to see that the flow $h_t^+$ is correctly defined for $\nu$-almost all points in $X$ and that $h_t^+$ preserves the measure $\nu$. Moreover, the flow $h_t^+$ is uniquely ergodic for $\mathbb{P}$-almost all Markov compacta. In the same way one can define \textit{a horizontal} flow $h_t^-$ along the leaves of the foliation $\EuScript{F}^-(X)$ using a reversed Vershik ordering.

	The first return map of $h_t^+$ to Markov compactum $\hat{X}$ obtained by deleting graphs $\Gamma_n, n\le 0$ is called \textit{Vershik automorphism}. It sends a path in a graph associated with $\hat{X}$ to its successor with respect to Vershik ordering (for one-sided Markov compacta the successor is almost surely correctly defined).  For a more detailed explanation see Bufetov ~\cite{Bufetov'13}.

	\textbf{Important remark.} In order to define the measure $\Phi_1^+$ (as well as other measures from the space $\mathfrak{B}^+(X)$) on arcs of type $[x, h_t^+x]$ properly we shall assume that the number of paths in $X$ grows at most sub-exponentially, or, in other words, the sequence of matrices $A_n$ of $X$ has sub-exponential growth, that is for any $\varepsilon > 0$ there exists some constant $C_\varepsilon > 0$ such that 
$$
	\sum\limits_{i,j=1}^m {(A_n)}_{ij} \le C_\varepsilon \exp (\varepsilon |n|) \text{ as }|n| \to \infty.
$$  	
	\\ 
	\par
	Let us consider the \textit{cohomological equation}
\begin{equation}\label{eqn}
	\left.\frac{d}{dt}\right|_{t=0} u(h_t^+x) = f(x).
\end{equation}
	Our aim is to find sufficient conditions for existence of solution of ~\eqref{eqn}.  

	A function $g$ is called Lipschitz on Markov compactum $X$, if it satisfies \textit{the weak Lipschitz property}, that is for any $n \in \mathbb{Z}$ and some constant $C>0$ depending only on $g$ we have
\begin{equation}\label{wlip}
\begin{matrix}
 \left| \int\limits_{\gamma_n^+(x)} gd\Phi_1^+ - \int\limits_{\gamma_n^+(y)} gd\Phi_1^+ \right| \le C &  \text{if $F(x_n) = F(y_n)$}. 
\end{matrix} 
\end{equation}
	A function $f$ on $X$ is said to belong to the class $C^{1+Lip}$ if it is continuous (in a weak topology of $X$) and the derivative $\varphi(x) = \left.\frac{d}{ds}\right|_{s=0} f(h_s^-x)$ along the horizontal direction is a Lipschitz function on $X$.  

\begin{theorem} Let $\mathbb{P}$ be a probability measure on the space $\Omega$ satisfying Assumption 1. Then for $\mathbb{P}$-almost every Markov compactum $X$ the following holds.   Let $h_t^+, h_s^-$ be the vertical and the horizontal flows on $X$, and let $f$ be a function of class $C^{1+Lip}$ on $X$.  Then there exists a bounded solution $u$ of ~\eqref{eqn} if and only if for any finitely-additive measure $\Phi^- \in \mathfrak{B}^-(X)$  we have $\int\limits_X fd\mathbf{m}_{\Phi^-} = 0$. 
\end{theorem}
	
	In other words, there is a finite number of obstacles to the existence of solution of \eqref{eqn}.

	Our result is closely related to ones of Marmi, Moussa and Yoccoz for interval exchange maps ~\cite{MMY'05} and of Forni for translation flows on flat surfaces ~\cite{Forni'97}.  
	Interval exchange maps, as well as translation flows, admit a symbolic representation as  Vershik automorphisms and their suspension flows respectively via Rauzy-Veech induction (see~\cite{Bufetov'13}) applied to Markovian partition into zippered rectangles.  
	The functionals $\mathbf{m}_{\Phi^-}$ then become invariant distributions in a sense of Forni ~\cite{Forni'97} with respect to the vertical flow $h_t^+$.

\section{Proof of Theorem 1}
	The rest of this paper will be devoted to the proof of Theorem 1. Due to Gottschalk-Hedlund theorem (see ~\cite{Johnson'78}, ~\cite{GH'55}), it is enough to prove that the integrals of the function $f$ along the arcs of the flow $h_t^+$ are uniformly bounded, that is the following
\begin{proposition} In conditions of Theorem 1 there exists some constant $C$ that does not depend on $T$ such that for some point $x_0 \in X$
\begin{equation}\label{bound}
	\left| \int\limits_0^T f \circ h_t^+(x_0) dt \right|  \le C	
\end{equation}
\end{proposition}

	The integral \eqref{bound} can be rewritten as $\int\limits_\gamma f d\Phi_1^+$, where $\gamma$ is an interval $[x_0, h_t^+ x_0]$ of the leaf $\gamma_\infty^+(x_0)$. Thus we have to prove that the integrals $\int\limits_\gamma f d\Phi_1^+$ are uniformly bounded for any arc $\gamma = [x_0, h_t^+ x_0]$.

\subsection{Difference between integrals for long arcs. \\} \text{ }

	Let us start with the proof of Proposition 1 for Markovian arcs, that is for arcs of the form $\gamma_n^+(x)$.

\begin{proposition} Under assumptions of Theorem 1
\begin{equation}\label{diffmark}
\begin{matrix} \left| \int\limits_{\gamma_n^+(x)} f d\Phi_1^+ - 
\int\limits_{\gamma_n^+(y)} f d\Phi_1^+ \right| \le C_\theta e^{-\theta n} & \text{if $F(x_n) = F(y_n)$},  
\end{matrix} \end{equation}
where $\theta = \theta_1 - \theta_2$, that is the difference between the top and the second Lyapunov exponents of the renormalization cocycle $\mathbb{A}$.
\end{proposition}
\begin{proof}
	Let us apply a Mean Value Theorem:
$$ \left| \int\limits_0^T f \circ h_t^+(x) dt - \int\limits_0^T f \circ h_t^+(y) dt \right| = \left| (x-y) \cdot \int\limits_0^T  \frac{\partial f}{\partial s}(h_t^+ h_{s_0}^- x) dt \right|= $$
$$ = \left|(x-y) \cdot \int\limits_0^T \varphi \circ h_t^+(z)dt \right| \le \Phi_1^-([x,y])\cdot \left|\int\limits_{\gamma_n^+(z)} \varphi d\Phi_1^+\right|, $$
where $\Phi_1^-$ is the positive measure on $\EuScript{F}^-(X)$.  

	We know that on smaller arcs the measure $\Phi_1^-$ decays exponentially with the top Lyapunov exponent $\theta_1$, that is $\Phi_1^- ([x,y]) \le C_{1} e^{-\theta_1 n}$ for some constant $C_1 > 0$. 
Thus it remains to prove that there exists a constant $C_2 > 0$ such that $\left|\int\limits_{\gamma_n^+(z)} \varphi d\Phi_1^+\right| \le C_2 e^{\theta_2 n}$.

	Let us choose canonical system of arcs $\gamma^+_{i,n} = \gamma_n^+(z_{i,n})$ and a sequence $\{v_n,  n \in \mathbb{Z} \}$ of vectors defined by the formula
\begin{equation}\label{canon}
\begin{matrix}
	{(v_n)}_i = \int\limits_ {\gamma_{i,n}^+}  \varphi d\Phi_1^+,& i\in\{1,...,m\}.
\end{matrix}
\end{equation}
	Since $\varphi$ is a Lipschitz function on $X$, and due to the sub-exponential growth of $A_n$ the following inequality holds:
$$
	|A_n v_n - v_{n+1}| \le C_\varepsilon e^{\varepsilon n}  
$$
for any $n\in\mathbb{N}$, $\varepsilon > 0$ and some $C_\varepsilon > 0$ not depending on $n$.  
	Indeed, $\varphi$ is Lipschitz, hence 
$$\left| \int\limits_{\gamma_n^+(z)} \varphi d\Phi_1^+ - 
\int\limits_{\gamma_n^+(\tilde{z})} \varphi d\Phi_1^+ \right| \le C  \text{ if $F(z_n) = F(\tilde{z}_n)$}. $$
Splitting the arc $\gamma_{i,n+1}$ into arcs of the form $\gamma_n^+(z)$ and replacing these arcs with the ones of the form $\gamma_{j,n}^+$ we get at most sub-exponential error due to Remark in Section 1.5.

	Thus, from Lemma 2.3 in  Bufetov ~\cite{Bufetov'13}, it follows that there exists a unique vector $\hat{v}_\varphi \in E_1^u$ such that for any $n \in \mathbb{N}$
$$ 
	|A_n...A_1 \hat{v}_\varphi - v_{n+1}| \le C_\varepsilon e^{\varepsilon n},
$$
	 and this yields that there exists a unique finitely-additive measure $\Phi_\varphi^+ \in \mathfrak{B}^+(X),$  corresponding to the vector $\hat{v}_\varphi$ such that for any $n \in \mathbb{N}$
\begin{equation}\label{Lip-approx}
	 \left| \int\limits_{\gamma_{i,n}^+} \varphi d\Phi_1^+ - \Phi_\varphi^+(\gamma_{i,n}^+) \right| \le C_\varepsilon e^{\varepsilon n},
\end{equation}
	where $C_\varepsilon > 0$ depends only on function $\varphi$.

	Moreover, for any functional $\Phi^- \in \mathfrak{B}^-(X)$ a pairing $\langle \Phi_\varphi^+, \Phi^- \rangle$ can be re-written as a Riemann-Stiltjes integral:
\begin{equation}
\langle \Phi_\varphi^+, \Phi^- \rangle = \int_X \varphi d\mathbf{m}_{\Phi^-}.
\end{equation}

	Both flows $h_t^+$ and $h_t^-$ preserve the measure $\nu = \Phi_1^+ \times \Phi_1^-$ on $X$. It yields $\int_X \varphi d\nu  = 0$.
Furthermore, applying formula \eqref{base} to $\Phi^+_\varphi$ we get a decomposition
$$
	\Phi_\varphi^+ = \langle \Phi_\varphi^+, \Phi_1^- \rangle \Phi_1^+ + ... + \langle \Phi^+_\varphi, \Phi_d^- \rangle \Phi_d^+.
$$
	The first term in this sum is equal to $\int_X \varphi d\nu \cdot \Phi_1^+$, so it equals to zero.  Thus there exists $C_2 >0$ that does not depend on $n\in\mathbb{N}$ such that
$$  \begin{matrix} \left|\int\limits_ {\gamma_{i,n}^+}  \varphi d\Phi_1^+\right| \le C_2 e^{\theta_2 n} &  \text{ for any $n\in \mathbb{N},$} \end{matrix}$$
 and this finishes the proof of Proposition 2.
\end{proof}

\subsection{ Proof of Proposition 1 for Markovian arcs. \\} \text{ }
 	
	Now we come to the new sequence of vectors associated with the integrals of $f$ along canonical arcs
\begin{equation}\label{canon-f}
\begin{matrix}
	{(w_n)}_i = \int\limits_ {\gamma_{i,n}^+}  f d\Phi_1^+,& i\in\{1,...,m\}.
\end{matrix}
\end{equation}

	Proposition 2 yields that for any $n \in \mathbb{N}$ and $\theta = \theta_1 - \theta_2$
$$
	\left| A_n w_n - w_{n+1} \right| \le C_{\theta} \exp(-\theta n) 
$$

\begin{proposition}
	In assumptions of Theorem 1, let $w_1,...,w_n, ...$ be a sequence of vectors such that there exist a constant $C_\theta > 0$ and $\theta > 0$ such that for all $n \ge 1$ we have
	$$ |A_n w_{n} - w_{n+1}| \le C_\theta \exp(-\theta n).$$
	Then there exists a unique vector $\hat{w} \in E_u^1$ such that	$$|A_n...A_1\hat{w} - w_{n+1}| \le C. $$ 
\end{proposition}
\begin{proof}
	Let $E_n^{cu}$ and $E_n^{s}$ be the central-unstable and stable subspaces of the cocycle $\mathbb{A}(n, X)$ respectively (their existence is provided by Oseledets Multiplicative Ergodic Theorem). For any $n \in \mathbb{N}$ a vector $w_{n+1} \in \mathbb{R}^m$ admits a direct-sum decomposition
$$
	w_{n+1}^+ = u_{n+1}^+ + A_n u_{n}^+ + ... + A_n...A_1u_{1}^+,
$$
$$
	w_{n+1}^- = u_{n+1}^- + A_n u_{n}^- + ... + A_n...A_1u_{1}^-,
$$
	where $w_{n+1}^+ \in E_{n+1}^{cu}$, $w_{n+1}^- \in E_{n+1}^{s}$ and $w_{n+1} = w_{n+1}^+  + w_{n+1}^-$.

	%A vector $$\hat{w}  = u_1^+ + A_1^{-1} u_2^+ + ... + (A_n...A_1)^{-1}u_{n+1}^+ +....$$ gives us an approximation up to sub-exponentional error:
%$$|A_n...A_1\hat{w} - w_{n+1}| \le C_\varepsilon e^{\varepsilon n},$$
%since $|A_n...A_1\hat{w} - w_{n+1}^+| \le C'$ and $|w_{n+1}^-| \le C''_\varepsilon e^{\varepsilon n}$, and we want to get reduce it to some constant not depending on $n$.
 
	Let us choose $\delta$ so that $\frac{\theta}{2} > \delta > 0$ and set $B_n =  e^{-\delta}A_n$, and let $\tilde{w}_n = e^{-\delta (n-1)} w_n$.  Again, for a vector $\tilde{u}_{n+1} = B_n \tilde{w}_n - \tilde{w}_{n+1}$ we have
$$
	|\tilde{u}_{n+1}| = |A_n w_n - w_{n+1}| \cdot e^{-\delta  n} \le C e^{-\tilde{\theta} n} 
$$ 
for some $C>0$, where $\tilde{\theta} = \theta + \delta > 0$. 
	
	Again we split the vectors $\tilde{u}_{n+1}$ into central-unstable and stable components $\tilde{u}_{n+1}^+$ and $\tilde{u}_{n+1}^-$ respectively, and
$$
	\tilde{w}_{n+1}^+ = \tilde{u}_{n+1}^+ + B_n \tilde{u}_{n}^+ + ... + B_n...B_1\tilde{u}_{1}^+,
$$
$$
	\tilde{w}_{n+1}^- = \tilde{u}_{n+1}^- + B_n \tilde{u}_{n}^- + ... + B_n...B_1\tilde{u}_{1}^-.
$$
	Note that for any $n \in \mathbb{N}$ we have $\tilde{u}_{n}^+ \in E_n^u$, hence $\tilde{w}_n^+ \in E_n^{u}$.
	Let us introduce a vector $\hat{w} \in E_1^u$:	
$$\hat{w}  = \tilde{u}_1^+ + B_1^{-1} \tilde{u}_2^+ + ... + (B_n...B_1)^{-1}\tilde{u}_{n+1}^+ +....$$
	It is not hard to check that we have the following inequalities: 
$$
	|B_n...B_1\hat{w} - \tilde{w}^+_{n+1}| \le C' \exp(-\tilde{\theta}n),
$$
$$ |\tilde{w}_{n+1}^-| < C'' \exp(-\delta n),$$
	and hence
$$
	|B_1...B_n \hat{w} - \tilde{w}_{n+1}| \le \tilde{C} e^{-\delta n}.
$$
Thus we have arrived at the required estimate:
$$
	|A_1...A_n \hat{w} - w_{n+1}| \le C.
$$ 
	  The uniqueness of the vector $\hat{w}$ is obvious from the fact that it belongs to unstable subspace of $\mathbb{A}$.
\end{proof}

	As a result, 
\begin{proposition} In assumptions of Theorem 1, there exists a unique measure $\Phi_f^+ \in \mathfrak{B}^+$ such that for any $n\in \mathbb{N}$
\begin{equation}
\left| \int\limits_{\gamma_n^+(z)} fd\nu^+ - \Phi_f^+ (\gamma_n^+(z)) \right| \le C (||\varphi||_{Lip^+_w} + ||f||_{C}),  \text{ where } C >0
\end{equation}
\end{proposition}
	Moreover, assumptions of Theorem 1 yield that $\Phi_f^+(\cdot)$ is equal to 0, since for each functional $\Phi^- \in \mathfrak{B}^-(X)$ we have $ \langle \Phi_f^+, \Phi^- \rangle = 0$. This concludes the proof of Propositon 1 for  Markovian arcs, that is for the sets of type $\gamma_n^+(\cdot), n \in \mathbb{Z}$.

\subsection{The integrals along arbitrary arcs are bounded. \\} \text{ }

	It is not too hard to see that the Proposition 1 is true for arbitrary arcs of the flow $h_t^+$. Indeed, let us take an arc $\gamma = [x_0,h_t^+x_0]$ of the vertical flow and approximate it with Markovian arcs.  The arc $\gamma$ admits a decomposition into markovian arcs
	\begin{equation}\label{decomposition}
	\gamma = [x, h_T^+x] = \bigsqcup_{n=-\infty}^{0} \bigsqcup_{k=1}^{N_n} \gamma_{n,k},  
	\end{equation}
where $ \gamma_{n,k}$ is arc of $n$-th level, and $N_n$ grows sub-exponentially as $n$ goes to $-\infty$. 
	Due to sub-exponential growth of $N_n$, we have
	\begin{equation}\label{geom-series}
	\left|\int\limits_0^T f \circ h_t^+(x) dt \right| \le \sum_{n \le 0} \exp(\varepsilon|n|)\cdot \left|\int\limits_{\gamma_{n,k}} fd\nu^+\right|. 
	\end{equation}
	And finally, due to the sub-exponential growth of $A_n$ for negative $n$, terms in the sum ~\eqref{geom-series} decay exponentially fast with the Lyapunov exponent $\theta_1$, and as a result, time integrals are uniformly bounded:
$$ \left|\int\limits_0^T f \circ h_t^+(x) dt \right| \le C_{f}. $$
Now Proposition 1 and then Theorem 1 are completely proved. 

\section*{Acknowledgements}
It is a great pleasure to thank Alexander I. Bufetov for stating the problem and encouraging me to write this paper. I am deeply grateful to Alexey V. Klimenko for many fruitful discussions. In particular, an idea of the proof of Proposition 3 came from discussions with him.

\end{document}